\documentclass[12pt,oneside]{amsart}
\usepackage{amscd,amsmath,amssymb,amsfonts,a4}
\usepackage{hyperref}
\usepackage{color}

\newlength{\rulebreite}

\def\timesover#1#2#3{\ \xymatrix@1@=0pt@M=0pt{ _{#1}&\times&_{#2} \\& ^{#3}&}\ }
\def\otimesover#1#2#3{\ \xymatrix@1@=0pt@M=0pt{ _{#1}&\otimes&_{#2} \\& ^{#3}&}\ }

\usepackage[all]{xy}
\theoremstyle{plain}
\newtheorem{thm}{Theorem}

\theoremstyle{definition}

\newtheorem{rmk}[thm]{Remark}

\newtheorem{claim}[thm]{Claim}
\numberwithin{thm}{section}
\numberwithin{equation}{section}

\newcommand{\ga}[2]{
\begin{gather}\label{#1}#2\end{gather} 
}

\newcommand{\surj}{\twoheadrightarrow}

\newcommand{\Spec}{{\rm Spec \,}}

\newcommand{\sC}{{\mathcal C}}

\newcommand{\sF}{{\mathcal F}}

\newcommand{\sL}{{\mathcal L}}
\newcommand{\sM}{{\mathcal M}}

\newcommand{\sO}{{\mathcal O}}
\newcommand{\sP}{{\mathcal P}}

\newcommand{\sT}{{\mathcal T}}

\newcommand{\sX}{{\mathcal X}}
\newcommand{\sY}{{\mathcal Y}}
\newcommand{\sZ}{{\mathcal Z}}
\newcommand{\F}{{\mathbb F}}

\newcommand{\N}{{\mathbb N}}
\renewcommand{\P}{{\mathbb P}}
\newcommand{\Q}{{\mathbb Q}}

\newcommand{\Z}{{\mathbb Z}}

\begin{document}

\title[Congruence and coniveau]{Congruence for rational points over finite fields and coniveau over local fields}
\author{H\'el\`ene Esnault}
\address{
Universit\"at Duisburg-Essen, Mathematik, 45117 Essen, Germany}
\email{esnault@uni-due.de}
\author{Chenyang Xu}
\address{Department of Mathematics, Princeton University, Princeton NJ 08544, USA}
\email{chenyang@math.princeton.edu}
\date{June 6, 2007}
\thanks{Partially supported by  the DFG Leibniz Preis and the Amercian Institute for Mathematics}
\begin{abstract}
If the $\ell$-adic cohomology of a projective smooth variety,
defined over a local field $K$ with finite residue field $k$, is supported in codimension $\ge 1$, 
then every model over the ring of integers of $K$ has a $k$-rational point. For $K$ a $p$-adic field, this is \cite[Theorem~1.1]{Ereg}. If the model $\sX$ is regular, one has a congruence $|\sX(k)|\equiv 1 $ modulo $|k|$ for the number of $k$-rational points (\cite[Theorem~1.1]{Ept}). The congruence is violated if one drops the regularity assumption. 

\end{abstract}
\maketitle
\section{Introduction}
Let $X$ be a projective variety defined over a local field $K$ with finite residue field $k=\F_q$. Let $R$ be the ring of integers of $K$. A {\it model} of $X/K$ is a flat projective morphism 
$\sX \to {\rm Spec}(R)$, with $\sX$ an integral scheme,  such that tensored with $K$ over $R$, it is $X\to {\rm Spec}(K)$. 
As in \cite{Ept} and \cite{Ereg},  we  consider $\ell$-adic cohomology $H^i(\bar X)$ with $\Q_\ell$-coefficents. Recall briefly that one defines the first coniveau level 
\ga{}{N^1H^i(\bar X)=\{\alpha \in H^i(\bar X), \exists \ {\rm divisor} \ D\subset X \ {\rm s.t.} \ 
0=\alpha|_{X\setminus D}\in H^i(\overline{X\setminus D})\}.\notag}
As $H^i(\bar X)$ is a  finite dimensional $\Q_\ell$-vector space, one has by localization
\ga{}{\exists D \subset X \ {\rm s.t.} \ N^1H^i(\bar X)= {\rm Im}\big(H^i_{\bar D}(\bar X)\to H^i(\bar X)\big),\notag}
where $D\subset X$ is a divisor. One says that $H^i(\bar X)$ is {\it supported in codimension 1} if $N^1H^i(\bar X)=H^i(\bar X)$. The purpose of this note is twofold. We show the following theorem. 
\begin{thm} \label{thm1.1}
 Let $X$ be a smooth,  projective, absolutely irreducible variety defined over a local field $K$   with finite residue field $k$. Assume that $\ell$-adic cohomology $H^i(\bar X)$
is supported in codimension $\ge 1$ for all $i\ge 1$. Let $\sX$ be a model of $X$ over the ring of integers $R$ of $K$. Then  there is a 
projective surjective morphism $\sigma: \sY\to \sX$ of $R$-schemes  such that 
 $$|\sY(k)|\equiv 1 \ {\rm mod} \ |k|.$$ 
In particular, any model $\sX/R$ of $X/K$ has a $k$-rational point. 
\end{thm}
This generalizes \cite[Theorem~1.1]{Ereg} where the theorem is proven under the assumption that $K$ has characteristic 0. On the other hand, assuming that $\sX$ is regular, we showed in \cite[Theorem~1.1]{Ept} that the number of $k$-rational points $|\sX(k)|$ is congruent to $1$  modulo $|k|$. It was in fact the way to show that $k$-rational points exist on $\sX$, as surely $|k|$, being a $p$-power, where $p$ is the characteristic of $k$, is $>1$. We show that if we drop the regularity assumption, there are models which, according to Theorem \ref{thm1.1}, have a rational point, but do not satisfy the congruence.
\begin{thm} \label{thm1.2}
Let $X_0=\P^2$ over $K_0:=\Q_p$ or $\F_p((t))$. Then there is a finite field extension $K\supset K_0$, which can be chosen to be unramified, and there is a normal model $\sX/R$ of $X:=X_0\otimes_{K_0} K$, such that $|\sX(k)|$ is not congruent to $1$ modulo $|k|$. 
\end{thm}
The proof of Theorem \ref{thm1.1}  follows closely the one in unequal characteristic in \cite[Theorem~1.1]{Ereg}, and, aside of Deligne's integrality theorem \cite[Corollaire~5.5.3]{DeInt} and \cite[Appendix]{Ept} and purity \cite{Fu}, relies strongly on de Jong's alteration theorem as expressed in 
\cite{deJ2}. However, we have to replace the trace argument we used there by a more careful analysis of the Leray spectral sequence stemming from de Jong's construction. 
The construction of the examples in Theorem \ref{thm1.2} uses Artin's contraction theorem as expressed in \cite{artin} and is somewhat inspired by Koll\'ar's construction exposed in 
\cite[Section~3.3]{BE}.\\ \ \\
{\it Acknowledgements:} We thank Johan de Jong for his interest. 
\section{Proof of Theorem \ref{thm1.1}} 
This section is devoted to the proof of Theorem \ref{thm1.1}. 

Let $K$ be a local field  with finite residue field $k$. Let $R\subset K$ be its valuation ring. 
Let $\sX\to \Spec R$ be a model of a projective variety $X\to \Spec K$. We do not assume here that $X$ is absolutely irreducible, nor do we assume that $X/K$ is smooth. Then by \cite[Corollary~5.15]{deJ2}, there is a diagram
\ga{2.1}{\xymatrix{\ar[drr] \sZ\ar[r]^{\pi} & \ar[dr] \sY \ar[r]^{\sigma}   & \sX \ar[d]\\
  & & \Spec R
}
}
and a finite group $G$ acting on $\sZ$ over $\sY$
with the properties
\begin{itemize}
\item[(i)] $\sZ\to \Spec R$ and $\sY\to \Spec R$ are  flat,
\item[(ii)] $\sigma$ is projective, surjective, $K(\sX)\subset K(\sY)$ is a purely inseparable field extension,
\item[(iii)] $\sY$ is the quotient of $\sZ$ by $G$,
\item[(iv)] $\sZ$ is regular. 
\end{itemize}
We want to show that this $\sY$ does it. 
Let us set  $$Y=\sY\otimes K, \ Z=\sZ\otimes K.$$
The only difference with \cite[(2.1)]{Ereg} is that $K(\sX)\subset K(\sY)$ may be a purely inseparable extension rather than an isomorphism. Thus, the argument there breaks down as
 one does not have  traces as in  \cite[(2.3), (2.4)]{Ereg}. We do not have \cite[(2.5)]{Ereg} a priori, and we can't conclude \cite[Claim~2.1]{Ereg}. 

Let us overtake the notations of {\it loc. cit.}: we endow all schemes considered (which are   $R$-schemes) with the upper subscript $^u$ to indicate the base change $\otimes_R R^u$ or $\otimes_K K^u$, where $K^u\supset K$ is the maximal unramified extension, and $R^u\supset R$ is the normalization of $R$ in $K^u$. Likewise, we write $\overline{?}$
to indicate the base change $\otimes_R \bar R, \ \otimes_K \bar K, \ \otimes_k \bar k$, where $\bar K\supset K, \ \bar k\supset k$ are the algebraic closures and $\bar R\supset R$ is the normalization of $R$ in $\bar K$. 
We consider as in \cite[(2.1)]{Ept} the $F$-equivariant exact sequence (\cite[3.6(6)]{DeWeII})
\ga{2.2}{\ldots \to H^i_{\bar B}(\sY^u)\xrightarrow{\iota} H^i(\bar B)=H^i(\sY^u)\xrightarrow{sp^u} H^i(Y^u) \to \ldots, 
}
where  $F \in {\rm Gal}(\bar k/k)$ is the geometric Frobenius, and $B=\sY\otimes k$.
We have \cite[Claim~2.2]{Ereg}  unchanged: 
\begin{claim} \label{claim2.1}
The eigenvalues of the geometric Frobenius $F\in {\rm Gal}(\bar k/k)$  acting on
$\iota(H^i_{\bar B}(\sY^u))\subset H^i(\bar B)$ lie in $q\cdot \bar{\Z}$ for all $i\ge 1$.
\end{claim}
So the problem is to show that the eigenvalues of $F$ acting on Im$(sp^u)\subset H^i(Y^u)$ lie in 
$q\cdot \bar{\Z}$ as well. Let us decompose the morphism $\sigma$ as 
\ga{2.3}{\sigma: Y\xrightarrow{\tau} X_1\xrightarrow{\epsilon} X}
where $X_1$ is the normalization of $X$ in $K(Y)$. Thus in particular, $\tau$ is birational, $\epsilon$ is finite and purely inseparable. Let us denote by $U\subset X$ a non-empty open such that $\tau|_{\epsilon^{-1}(U)}$ is an isomorphism, and let us set $D:= X\setminus U$.
We define
\ga{2.4}{\sC:={\rm cone}(\Q_\ell \to R\tau_*\Q_\ell)[-1]}
as an object in the bounded derived category of $\Q_\ell$-constructible sheaves on $X_1$. 
Since $\tau_*\Q_\ell=\Q_\ell$, the cohomology sheaves of $\sC$ are in degree $\ge 1$, and have support in $D_1:=D\times_X X_1$. We conclude
\ga{2.5}{H^i_{D_1^u}(X_1^u, \sC)=H^i(X_1^u, \sC) \ \forall i\ge 0.}
One  has the commutative diagram of exact sequences
\ga{2.6}{\xymatrix{ H^{i+1}_{ D^u_1}( X_1^u) \\ 
\ar[u] H^i_{{ D_1^u}}( X_1^u, \sC) \ar[r]^{= \eqref{2.5}} &
H^i(X_1^u, \sC) \\
\ar[u] H^i_{E^u}( Y^u) \ar[r] & \ar[u] H^i(Y^u)  \\
\ar[u] H^i_{D_1^u}(X_1^u) \ar[r] & \ar[u] H^i(X_1^u) 
}}
where $E= \sigma^{-1}(D)$. By \cite[Theorem~1.5~and~Appendix]{Ept}
 the eigenvalues of $F$ on $H^i( X^u)=H^i(X_1^u)$
and on $H^{i+1}_{ D^u_1}( X_1^u)=H^{i+1}_{D^u}(X^u)$ 
 lie in $q\cdot \bar{\Z}$.  For the latter cohomology, one has to argue again by purity on 
$X^u$ before applying {\it loc. cit.}:
by purity one is reduced to considering cohomology of the type $H^a(\Sigma^u)(-1)$ for a regular scheme $\Sigma$ over $K$ and $a\ge 0$.
 It remains to consider the eigenvalues of $F$ on $H^i_{E^u}( Y^u) =H^i_{L^u}(Z^u)^G$ where $L=D\times_X Z$. This is again the argument by purity and then  {\it loc. cit.}
So we conclude
\begin{claim} \label{claim2.2}
 The eigenvalues of the geometric Frobenius $F\in {\rm Gal}(\bar k/k)$  acting on $H^i(Y^u)$, and therefore on ${\rm Im}(sp^u)\subset H^i(Y^u)$, lie in $q\cdot \bar{\Z}$ for all $i\ge 1$.
\end{claim}
So we conclude now as usual that all the eigenvalues of $F$ acting on $H^i(\bar B)$ lie in $q\cdot  \bar{\Z}$ for $i\ge 1$, thus the Grothendieck-Lefschetz
trace formula applied to $H^*(\bar B)$, together with the absolute irreducibility of $B$,  imply the congruence. 
This finishes the proof of Theorem \ref{thm1.1}. 
\section{Construction of examples}
This section is devoted to the proof of Theorem \ref{thm1.2}.

Let us first recall that if $E$ is a smooth genus 1 curve over a finite field $\F_q$, it is always an elliptic curve, which means that it always carries a $\F_q$-rational point.  Furthermore one has
\begin{claim} \label{claim3.1} 
Given an elliptic curve $E/\F_q$, there is a finite field extension $\F_{q^n}\supset \F_q$ such that $|E(\F_{q^n})|$ is not congruent to $1$ modulo $q^n$. 

\end{claim}
\begin{proof}
By the trace formula, 
$|E(\F_{q^n})|$  being congruent to $1$ modulo $q^n$ for all $n\ge 1$ is equivalent to saying that the eigenvalues of $F^n$ acting on $H^i(\bar E)$ lie in $q^n\cdot \bar{\Z}$ for all $n\ge 1$ and $i\ge 1$. 
By purity (which in dimension 1 is Weil's theorem), this is equivalent to saying that the 
eigenvalues of $F^n$ acting on $H^1(\bar E)$ lie in $q^n\cdot \bar{\Z}$ for all $n\ge 1$. On the other hand, by duality, if $\lambda$ is an eigenvalue, then $\frac{q^n}{\lambda}$ is also an eigenvalue. This is then impossible that both $\lambda$ and 
$\frac{q^n}{\lambda}$ be $q^n$-divisible as algebraic integers. 

\end{proof}
We now construct the following scheme. Let us set $\sP_0:=\P^2$ over $R_0:=\Z_p$ or over $\F_p[[t]]$. Choose an elliptic curve $E_0 \subset \sP\otimes \F_p=\P^2_{\F_p}$ defined over $\F_p$. Let $k\supset \F_p$ be a finite field extension such that $|E_0(k)|$ is not $k$-divisible (Claim \ref{claim3.1}). Set $E:=E_0\otimes_{\F_p} k, \ \sP:=\sP_0\otimes_{R_0} R$, with $R=W(k)$ or $\F_q[[t]]$, and $K={\rm Frac}(R)$. 
Choose a smooth projective curve  $\sC\subset \sP$ over $R$, of degree $\ge 4$, such that $C:=\sC\otimes k$ is transversal to $E$. Define $\Sigma=E\cap C \subset E$ to be the 
 0-dimensional intersection subscheme. It has degree $\ge 12$, thus in particular $>9$.
Let $b: \sY\to \sP$ be the blow up of $\Sigma \subset \sP$. We denote by $P_\Sigma$ the exceptional locus, which is a trivial $\P^2$ bundle over $\Sigma$, by $Y$ the strict transform of $\P^2_{k}$, and we still denote by $E\subset Y$ the strict transform of the elliptic curve. Then the conormal bundle $N^\vee_{E/\sY}$ of $E$ in $\sY$ is an extension of the conormal bundle 
$N^\vee_{E/Y}$ of $E$ in $Y$ by the restriction to $E$ of the conormal bundle $N^\vee_{Y/\sY}$ of $Y$ in $\sY$, both ample line bundles on $E$ by the condition on the degree of $\Sigma$. 

Let $I\subset \sO_{\sY}$ be the ideal sheaf of $E$. For a coherent sheaf $\sF$ on $\sY$, we denote by $I^n/I^{n+1}\cdot \sF$ the image of $I^n/I^{n+1}\otimes_{\sO_{\sY}} \sF$ in $\sF$, where $n\in \N$.
\begin{claim}\label{claim3.2}
For every coherent sheaf $\sF$ on $\sY$, one has $H^1(E, I^n/I^{n+1}\cdot \sF)=0$ for all $n\in \N$ large enough. 
\end{claim}
\begin{proof}
As by definition one has a surjection $I^n/I^{n+1}\otimes_{\sO_{\sY}} \sF \surj 
I^n/I^{n+1}\cdot \sF$, it is enough to show $H^1(E, I^n/I^{n+1}\otimes_{\sO_{\sY}} \sF)=0$ for $n$ large enough. As $I^n/I^{n+1}$ is locally free, $I^n/I^{n+1}\otimes_{\sO_{\sY}} \sF$ is an extension of $I^n/I^{n+1}\otimes_{\sO_{\sY}} \sF_0$ by 
$I^n/I^{n+1}\otimes_{\sO_{\sY}} \sT$, where $\sT\subset \sF$ is the maximal torsion subsheaf and 
$\sF_0=\sF/\sT$ is locally free. As $H^1(E, I^n/I^{n+1}\otimes_{\sO_{\sY}} \sT)=0$, we may assume that $\sF$ is locally free. 
As $I^n/I^{n+1}$ is a locally free filtered sheaf, with associated graded a sum of ample line bundles of strictly increasing degree as $n$ grows, we have $H^1(E, {\rm gr}(I^n/I^{n+1})\otimes_{\sO_{\sY}} \sF)=0$ for $n$ large enough, and thus $H^1(E, I^n/I^{n+1}\otimes_{\sO_{\sY}} \sF)=0$ as well. 

\end{proof}
Artin's contraction criterion 
\cite[Theorem~6.2]{artin}
applied to $E\to {\rm Spec}(k)$, together with Artin's existence theorem \cite[Theorem~3.1]{artin} show the existence of a contraction
\ga{3.1}{a_1: \sY\to \sX_1}
where $\sX_1$ is an algebraic space over $R$, $a_1|_{\sY\setminus E}$ is an isomorphism and $a_1(E)={\rm Spec}(k)$. Let $\sX\xrightarrow{\nu} \sX_1$ be the normalization of $\sX_1$ in 
$K(\sY)=K(\sP)$. This is a normal  algebraic space over $R$. One has a diagram 
\ga{3.2}{\xymatrix{\ar[d]_b \sY \ar@/^1pc/[rr]^{a_1} \ar[r]_a & \sX \ar[r]_\nu & \sX_1\\
\sP
}
}
\begin{claim} \label{claim3.3}
$|\sX(k)|$ is not congruent to $1$ modulo $|k|$. 
\end{claim}
\begin{proof}
By \cite[Theorem~1.1]{Ept} (or by a simple computation in this case),  $|\sY(k)|$ is congruent to $1$ modulo $|k|$.  By Claim \ref{claim3.1} and the choice of $E$, $|\sX_1(k)|$ is not congruent to $1$ modulo $|k|$. On the other hand, as the fibers of $a_1$ are absolutely irreducible, $\nu$ has to be a homeomorphism. Thus $|\sX(k)|=|\sX_1(k)|$. This finishes the proof.  

\end{proof}
In order to finish the proof of Theorem \ref{thm1.2}, it remains to show
\begin{claim} \label{claim3.4}
$\sX\to {\rm Spec}(R)$ is a model of $X=\P^2/K$. 
\end{claim}
\begin{proof}
We have to show that $\sX\to {\rm Spec}(R)$ is a flat projective morphism. 
Since $\sX$ is reduced, ${\rm Spec}(R)$
is regular of dimension 1,  then  \cite[IV~Proposition~14.3.8]{EGA} allows to conclude that $\sX/R$ is flat. Thus we just have   to show
that $\sX/R$ is projective.  
To this aim, 
we want to decend a line bundle from $\sY$ to $\sX$. 
Let us define the line bundle $\sM:=b^*\sO_{\sP}(\sC)(-P_\Sigma)$ on  $\sY$. By definition, one has
\ga{3.3}{ \sM|_E\cong \sO_E.}
\begin{claim} \label{claim3.5}
The line bundle $\sM$ descends to $\sX$, that is 
 there is a line bundle $\sL$ on $\sX$ with $a^*\sL=\sM$. 
\end{claim}
\begin{proof}[Proof of Claim \ref{claim3.5}]
The proper morphism of algebraic spaces $a: \sY\to \sX$, with $a_*\sO_{\sY}=\sO_{\sX}$, has the property that $a^{-1}a(E)=E$ set-theoritically, that $a|_{\sY\setminus E}: \sY\setminus E\to \sX\setminus a(E)$ is an isomorphism, and that $H^1(E, I^n/I^{n+1})=0$ for $n\ge 1$. So Keel's theorem \cite[Lemma~1.10]{keel} asserts that some positive power $\sM^{\otimes r}$ descends to $\sX$ if the following condition is fulfilled
\ga{3.4}{\forall m>0, \exists r(m)>0 \ {\rm s.t}  \ \sM^{\otimes r(m)}|_{E_m} \ {\rm descends \ to} \ 
a(E_m) \\ {\rm where}  \ E_m:={\rm Spec}(\sO_{\sY}/I^{m+1}). \notag}
So we just have to check that \eqref{3.4} is fulfilled with $r=1$ in our situation. 
The scheme $a(E_m)$ has Krull dimension $0$. 
Thus by Hilbert 90's theorem (see e.g. \cite[Corollary~11.6]{Milne}) one has
\ga{3.5}{{\rm Pic}(a(E_m))=0.} 
 We conclude that  to check \eqref{3.4} is equivalent to checking that $\sM^{\otimes r(m)}|_{E_m}\cong \sO_{E_m}$ for some positive power $r(m)$.  In fact one has
\ga{3.6}{\sM|_{E_m}\cong \sO_{E_m} \ \forall m\ge 1.}
For $m=1$, this is \eqref{3.3}. 
We argue by induction and assume that for $m>1$, we have a trivializing section  $s_m: \sO_{E_m}\xrightarrow{\cong} \sM|_{E_m}$. 
We want to show that it lifts to a trivializing section $s_{m+1}: \sO_{E_{m+1}}\xrightarrow{\cong} \sM|_{E_{m+1}}$.  
One has an exact sequence 
\ga{3.7}{0\to I^{m+1}/I^{m+2}\to \sM|_{E_{m+1}}\to \sM|_{E_m}\to 0.}
Since $H^1(E, I^{m+1}/I^{m+2})=0$, as $m\ge 0$, the trivializing section of $s_m: \sO_{E_m}\xrightarrow{\cong}  \sM|_{E_m}$ lifts to a section  
$s_{m+1}: \sO_{E_{m+1}}\to  \sM|_{E_{m+1}}$, and likewise, its inverse $t_m: \sM|_{E_m}\xrightarrow{\cong} \sO_{E_m}$ lifts to 
$t_{m+1}:  \sM|_{E_{m+1}} \to \sO_{E_{m+1}}$. The composite 
$t_{m+1}\circ s_{m+1}: \sO_{E_{m+1}} \to \sO_{E_{m+1}}$ lifts the identity of $\sO_{E_m}$. Therefore it is invertible. This shows that 
 $s_{m+1}$ trivializes. 
 The proof of Keel's theorem (see (2) after 
\cite[(1.10.1)]{keel}) shows then that one can take $r=1$.

\end{proof}
In order the finish the proof of Claim \ref{claim3.4}, 
it remains to see that $\sL$ on $\sX$ is ample. First, $\sL|_{\sX\otimes k}$ is ample because by \cite[Corollary~0.3]{keel}, this is enough to see that the linear system associated to  $\sL|_{\sX\otimes k}$ does not contract any curve, which is true by construction. So by Serre vanishing theorem, for sufficiently large $m$, $H^1(\sX\otimes k, \sL|_{\sX\otimes k}^{\otimes m})=0$. Base change implies  $H^1(\sX,\sL^{\otimes m})\otimes k=0$ 
(\cite[III~Theorem~7.7.5]{EGA}), thus by  Nakayama's lemma, one has 
\ga{3.8}{H^1(\sX,\sL^{\otimes m})=0 \ {\rm for \ } m \  {\rm large \ enough}.} 
As $\sL$ is invertible, multiplication $\sL^{\otimes m}\xrightarrow{\pi} \sL^{\otimes m}$  by the uniformizer $\pi$  is injective, with quotient $\sL|_{\sX\otimes k}^{\otimes m}$. Thus \eqref{3.8} implies surjectivity
$H^0(\sX, \sL^{\otimes m})\surj H^0(\sX\otimes k, \sL|_{\sX\otimes k}^{\otimes m})$ for $m$ large enough. Thus $H^0(\sX, \sL^{\otimes m})$ is a free $R$-module, and the linear system $H^0(\sX, \sL^{\otimes m})$ maps base point freely $\sX$ to $\P^N_R$, with $N+1={\rm rank}_R H^0(\sX, \sL^{\otimes m})$.
As it embedds $\sX\otimes k$, it embedds $\sX$ as well. 
This finishes the proof.

\end{proof}

\section{Dimension 1}
\begin{rmk} \label{rmk4.1}
 In Theorem \ref{thm1.1}, if $X/K$ has dimesnion 1, which means concretely if $X/K=\P^1/K$, then 
any normal model $\sX/R$ satisfies the congruence $|\sX(k)| \equiv 1$ modulo $|k|$. Thus the examples of Theorem \ref{thm1.2} have the smallest possible dimension.
\end{rmk}
\begin{proof}
 Indeed, using \eqref{2.1}, the only thing to check is that $H^1(\bar A)$, which is equal to $H^1(\sX^u)$, injects via $\sigma^*$ into $H^1(\bar B)=H^1(\sY^u)$.  Here 
 $A:=\sX\otimes_R k$. Let us denote by $\sX'$ the normalization of $\sX$ in $K(\sY)$, with factorization 
\ga{4.1}{\xymatrix{ \sY \ar@/^1pc/[rr]^{\sigma} \ar[r]_{\sigma'} & \sX' \ar[r]_\nu & \sX
}
}
and set $A':=A\times_{\sX} \sX'$. Then $\sigma'$ induces an isomorphism  $K(\sX')\xrightarrow{\cong} K(\sY)$. Furthermore, 
$\sX'\xrightarrow{\nu} \sX$ and and $A'\xrightarrow{\nu|_A} A$ are homeomorphisms. Thus 
$H^1(\sX^u)=H^1(\bar A)\xrightarrow{\nu^*} H^1((\sX')^u) =H^1(\bar{A'})$ is an isomorphism. 
On the other hand, since $\sigma'_*\Q_\ell=\Q_\ell$, the Leray spectral sequence for $\sigma'$ applied to  $H^1(\sY^u)$ yields an inclusion $H^1((\sX')^u)=H^1(\bar{A'})\xrightarrow{{\rm inj}} H^1(\sY^u)=H^1(\bar B)$. 
This finishes the proof.
\end{proof}

\bibliographystyle{plain}
\renewcommand\refname{References}

\end{document}